\documentclass[a4paper,reqno,11pt,oneside]{amsart}
\usepackage[left=2.7cm,right=2.7cm,top=3.5cm,
bottom=3.5cm]{geometry}
\usepackage[colorlinks=true,urlcolor=blue,
citecolor=red,linkcolor=blue,linktocpage,pdfpagelabels,
bookmarksnumbered,bookmarksopen]{hyperref}
\usepackage[english]{babel}
\usepackage{graphicx}
\usepackage{mathrsfs}
\usepackage{amssymb, amsmath, amsfonts, amsthm, mathtools}
\usepackage{xcolor}
\usepackage[shortlabels]{enumitem}
\usepackage{latexsym}
\usepackage{lineno}
\usepackage{appendix}

\makeatletter
\newtheorem*{rep@theorem}{\rep@title}
\newcommand{\newreptheorem}[2]{%
\newenvironment{rep#1}[1]{%
 \def\rep@title{#2 \ref{##1}}%
 \begin{rep@theorem}}%
 {\end{rep@theorem}}}
\makeatother

\allowdisplaybreaks

\makeatletter
\@namedef{subjclassname@2020}{%
  \textup{2020} Mathematics Subject Classification}
\makeatother

\def\po{{\partial \Omega}}
\def\H{{\mathcal H(\Omega)}}

\numberwithin{equation}{section}

\theoremstyle{plain}
\newtheorem{theorem}{Theorem}[section]
\newreptheorem{theorem}{Theorem}
\theoremstyle{plain}
\newtheorem{prop}[theorem]{Proposition}
\newreptheorem{prop}{Proposition}
\theoremstyle{plain}
\newtheorem{lemma}[theorem]{Lemma}
\newreptheorem{lemma}{Lemma}
\theoremstyle{plain}

\theoremstyle{definition}

\theoremstyle{definition}

\theoremstyle{definition}
\newtheorem{remark}[theorem]{Remark}
\theoremstyle{definition}

\theoremstyle{plain}

\theoremstyle{definition}

\begin{document}

\renewcommand{\labelenumi}{\textit{(\roman{enumi})}}

\title[Perturbed Schrödinger-Bopp-Podolsky system]{Existence of solutions with prescribed frequency for the perturbed Schrödinger-Bopp-Podolsky system in bounded domains}

\author[Danilo Gregorin Afonso]{Danilo Gregorin Afonso}
\address[Danilo Gregorin Afonso]{Dipartimento di Scienze Pure e Applicate \\
 Università degli Studi di Urbino Carlo Bo \\ Piazza della Repubblica 13, 61029 Urbino, Italy}
\email{danilo.gregorinafonso@uniurb.it}

\author[Bruno Mascaro]{Bruno Mascaro}
\address[Bruno Mascaro]{Faculdade de Computação e Informática \\ Universidade Presbiteriana Mackenzie \\ R. da Consolação 930, São Paulo, Brasil}
\email{bruno.mascaro@mackenzie.br}



\date{\today}

\begin{abstract}
    In this paper, we show that the Schrödinger-Bopp-Podolsky system with Dirichlet boundary conditions in a bounded domain possesses infinitely many solutions of prescribed frequency, for any set of (continuous) boundary conditions, provided that the Schrödinger equation is perturbed with a suitable nonlinearity. Our approach is variational, and our proof is based on a symmetric variant of the Mountain Pass theorem.
\end{abstract}

\maketitle

\section{Introduction}
\label{sec:intro}

In this paper, we are interested in analysing the existence of solutions to the so-called Schrödinger-Bopp-Podolsky system
\begin{equation}
    \label{eq:pde}
    \left\{
    \begin{array}{rcll}
        - \frac{1}{2} \Delta u + \phi u - g(x, u) & = & \omega u & \quad \text{ in } \Omega \\
        - \Delta \phi + \Delta^2 \phi & = & 4 \pi u^2 & \quad \text{ in } \Omega
    \end{array}
    \right.
    ,
\end{equation}
where $\Omega$ is a smooth bounded domain in $\mathbb R^3$, $g$ is a suitable nonlinearity and $\Delta^2 \phi = \Delta (\Delta \phi)$ is the bi-Laplacian operator. We consider Dirichlet boundary conditions, i.e.,
\begin{equation}
    \label{eq:Dirichlet_BC}
    \left\{
    \begin{array}{rcll}
        u & = & 0 & \quad \text{ on } \po \\
        \phi & = & h_1 & \quad \text{ on } \po \\
        \Delta \phi & = & h_2 & \quad \text{ on } \po \\
    \end{array}
    \right.
    ,
\end{equation}
and assume, for simplicity, that $h_1, h_2 \in C(\po)$. We refer to \cite{GazzolaGrunauSweers2010} for a discussion of appropriate boundary operators for higher-order elliptic problems in bounded domains.

The system of equations \eqref{eq:pde} models the (stationary) interaction of a charged particle with an electromagnetic field with the ansatz that the wave function is of the form
\begin{equation*}
    \psi(x, t) = u(x) e^{i \omega t},
\end{equation*}
where $u$ plays the role of the amplitude of the wave and $\omega$ is the frequency. To our knowledge, the first variational analysis of this kind of system appeared in \cite{dAveniaSiciliano2019} (see also \cite{QuoirinSicilianoSilva2024, MascaroSiciliano2022}) in the case of the whole space $\mathbb R^3$. In fact, \eqref{eq:pde} is a refinement of the much more studied Schrödinger-Maxwell system, introduced in \cite{BenciFortunato1998} (see also \cite{PisaniSiciliano2007b,PisaniSiciliano2007,PisaniSiciliano2013} and the references therein), where the second equation is just $- \Delta \phi = 4\pi u^2$. Besides the physical motivation, which consists of trying to overcome the so-called infinity problem of classical Maxwell theory (we refer to \cite{dAveniaSiciliano2019} for more on the physical aspects of the problem), the addition of the bi-Laplacian in the second equation gives rise to many interesting mathematical phenomena also when one considers boundary value problems in bounded domains, see e.g. \cite{AfonsoSiciliano2021}.

The main approaches developed to treat \eqref{eq:pde} are variational, and a choice is to be made. One can consider a normalization condition of the type 
\begin{equation*}
    \int_\Omega u^2 \ dx = 1,
\end{equation*}
in which case the parameter $\omega$ appears as a Lagrange multiplier (as done, e.g., in \cite{BenciFortunato1998, PisaniSiciliano2013, AfonsoSiciliano2021}), or one can perform a parametric analysis on the values of $\omega$ for which the free problem has a solution (see \cite{PisaniSiciliano2007b}). 

In this work, we follow the second approach. We are able to show that, provided that we perturb the Schrödinger equation with a suitable nonlinearity, then for any prescribed frequency $\omega$ and any set of continuous boundary conditions $h_1$ and $h_2$, the system \eqref{eq:pde}-\eqref{eq:Dirichlet_BC} possesses infinitely many solutions $(u, \phi)$ (see Theorem \ref{thm:main_Dirichlet}).

This work is organized as follows. In Section \ref{sec:prelim} we collect the main notations and definitions, and recall an important result that will be useful in our proofs. Section \ref{sec:existence} is devoted to the variational analysis of the perturbed problem and the proof of our main result, Theorem \ref{thm:main_Dirichlet}. We address the question of non-existence for the unperturbed problem in Section \ref{sec:nonexistence}.

\section{Preliminaries}
\label{sec:prelim}

\subsection{Notations and definitions}

Throughout the paper, $\Omega$ is a smooth, bounded domain (connected open set). For $1 \leq p \leq \infty$, $\|\cdot\|_p$ denotes the $L^p$ norm (whether on $\Omega$ or $\po$ will be clear from the context). As usual, we denote by $H_0^1(\Omega)$ the completion of $C_c^\infty(\Omega)$ with respect to the Sobolev norm $W^{1, 2}(\Omega)$. However, we consider $H_0^1(\Omega)$ with the equivalent norm
\begin{equation*}
    \|u\| \coloneqq \|\nabla u\|_2, \quad u \in H_0^1(\Omega).
\end{equation*}
Its dual space is denoted by $H^{-1}(\Omega)$.

The eigenvalues of $- \Delta$ in $H_0^1(\Omega)$ (counted with multiplicity) are denoted by $\lambda_k$, with $k \in \mathbb N$. The corresponding eigenspaces are denoted by $H_k$.

The Sobolev space $W^{2, 2}(\Omega)$ is, as usual, denoted by $H^2(\Omega)$. We also consider the functional space
\begin{equation*}
    \H \coloneqq H^2(\Omega) \cap H_0^1(\Omega)
\end{equation*}
endowed with the equivalent norm 
\begin{equation*}
    \|\varphi\|_\H \coloneqq \|\Delta \varphi\|_2, \quad \varphi \in \H.
\end{equation*}

Recall that a $C^1$ functional $J$ defined in a Banach space $E$ is said to satisfy the Palais-Smale condition if any sequence $(u_n)_{n \in \mathbb N}$ for which $(J(u_n))_{n \in \mathbb N}$ is bounded and $J'(u_n) \to 0$ as $n \to \infty$ possesses a convergent subsequence.

In our proofs, we perform some estimates and denote by $c_j$, $j \in \mathbb N$, some positive constants appearing in these estimates and whose exact values are not of interest.

On the nonlinearity $g \in C(\overline \Omega \times \mathbb R)$ appearing in \eqref{eq:pde} we assume the following:
\begin{enumerate}[label={($g_\arabic*$)}]
    \item \label{g1} $g$ is anti-symmetric: $g(x, - \xi) = - g(x, \xi)$ for all $x \in \overline \Omega, \xi \in \mathbb R$;

    \item \label{g2} $g$ satisfies
    \begin{equation*}
        \lim_{\xi \to 0} \frac{g(x, \xi)}{\xi} = 0 \quad \text{uniformly in } x;
    \end{equation*}

    \item \label{g3} there exist constants $a_1, a_2 \geq 0$ and $p \in (4, 6)$ such that for any $x \in \overline \Omega$ and $\xi \in \mathbb R$ it holds
    \begin{equation*}
        |g(x, \xi)| \leq a_1 + a_2 |\xi|^{p - 1};
    \end{equation*}

    \item \label{g4} there exist $r > 0$ and $\mu > 4$ such that for any $\xi \in \mathbb R$ with $|\xi| \geq r$ and any $x \in \overline \Omega$ it holds
    \begin{equation*}
        0 \leq \mu G(x, \xi) \leq \xi g(x, \xi),
    \end{equation*}
    where
    \begin{equation*}
        G(x, \xi) = \int_0^\xi g(x, t) \ dt.
    \end{equation*}
\end{enumerate}

\subsection{An abstract critical point theorem}

Here, we recall a version of the Mountain Pass Theorem for functionals that are symmetric with respect to the action of the group $\mathbb Z_2 = \{- \text{id}, \text{id}\}$ that will be useful in our proof.

\begin{theorem}[{\cite[Theorem 9.12]{Rabinowitz1986}}]
    \label{thm:Mountain_Pass}
    Let $E$ be an infinite-dimensional Banach space, $J \in C^1(E)$ be an even functional such that $J(0) = 0$ and $J$ satisfies the Palais-Smale condition. Suppose that $E = V \oplus X$, where $V$ is finite-dimensional and $J$ satisfies
    \begin{enumerate}
        \item There exist constants $\rho, \alpha > 0$ such that
        \begin{equation*}
            J_{B_\rho \cap X} \geq \alpha;
        \end{equation*}

        \item for each finite dimensional subspace $\widetilde E$ of $E$, there exists an $R = R(\widetilde E)$ such that $J \leq 0$ in $E \setminus B_{R(\widetilde E)}$.
    \end{enumerate}
    Then $I$ possesses an unbounded sequence of critical values, and therefore there exist infinitely many critical points.
\end{theorem}

\section{Existence of solutions}
\label{sec:existence}
\subsection{Variational framework}

In order to perform a variational analysis of the problem, it is convenient to slightly modify our system so as to make all boundary conditions homogeneous. To this aim, we consider the following auxiliary problem:
\begin{equation}
    \label{eq:auxiliary_problem_Dirichlet}
    \left\{
    \begin{array}{rcll}
        - \Delta \chi + \Delta^2 \chi & = & 0 & \quad \text{ in } \Omega \\
        \chi & = & h_1 & \quad \text{ on } \po \\
        \Delta \chi & = & h_2 & \quad \text{ on } \po
    \end{array}
    \right.
    .
\end{equation}

\begin{lemma}
    \label{lem:exitence_auxiliary_problem}
    Let $h_1, h_2 \in C(\Omega)$. Then there exists a weak solution $\chi \in H^2(\Omega)$ to \eqref{eq:auxiliary_problem_Dirichlet}. Moreover, $\chi \in C^4(\Omega) \cap C(\overline \Omega)$.
\end{lemma}

\begin{proof}
    Notice that substituting $\theta = \Delta \chi$ we obtain
    \begin{equation*}
        - \Delta \chi + \Delta^2 \chi = \Delta \theta - \theta = 0 \quad \text{ in } \Omega.
    \end{equation*}
    Now, by well-known results of linear elliptic equations (see, e.g., \cite{Brezis2010}), the problem
    \begin{equation*}
        \left\{ 
        \begin{array}{rcll}
            - \Delta \theta + \theta & = & 0 & \quad \text{ in } \Omega \\
            \theta & = & h_2 & \quad \text{ on } \po
        \end{array}
        \right.
    \end{equation*}
    admits a (unique) weak solution $\theta \in H^1(\Omega)$. Moreover, standard regularity estimates (see, e.g., \cite{GilbargTrudinger2001}) yield $\theta \in C^2(\Omega) \cap C(\overline \Omega)$.

    It is also well-known that the problem
    \begin{equation*}
        \left\{
        \begin{array}{rcll}
            - \Delta \chi & = & \theta & \quad \text{ in } \Omega \\
            \chi & = & h_1 & \quad \text{ on } \po
        \end{array}
        \right.
    \end{equation*} 
    admits a unique weak solution $\chi \in H^1(\Omega)$. Moreover, regularity theory yields $\chi \in C^4(\Omega) \cap C(\overline \Omega)$.

    The proof is complete, since $\chi$ thus found satisfies \eqref{eq:auxiliary_problem_Dirichlet}.
\end{proof}

Next, we make the change of variables
\begin{equation*}
    % \label{eq:def_\varphi}
    \varphi = \phi - \chi,
\end{equation*}
in such a way as to write the system in the variables $(u, \varphi)$:
\begin{equation}
    \label{eq:modified_system_Dirichlet}
    \left\{
    \begin{array}{rcll}
        - \frac{1}{2} \Delta u + (\varphi + \chi)u - g(x, u) & = & \omega u & \quad \text{ in } \Omega \\
        - \Delta \varphi + \Delta^2 \varphi & = & 4 \pi u^2 & \quad \text{ in } \Omega \\
        u & = & 0 &\quad \text{ on } \po \\
        \varphi & = & 0 & \quad \text{ on } \po \\
        \Delta \varphi & = & 0 & \quad \text{ on } \po
    \end{array}
    \right.
    .
\end{equation}

Let us consider the functional $F_\omega: H_0^1(\Omega) \times \H \to \mathbb R$ given by
\begin{align}
    F_\omega(u, \varphi)
    & = \frac{1}{4} \int_\Omega |\nabla u|^2 \ dx + \frac{1}{2} \int_\Omega (\varphi + \chi - \omega) u^2 \ dx - \int_\Omega G(x, u) \ dx \nonumber \\
    & \quad - \frac{1}{16 \pi} \int_\Omega |\nabla \varphi|^2 \ dx - \frac{1}{16 \pi} \int_\Omega |\Delta \varphi|^2 \ dx. \nonumber
\end{align}

All terms of this functional, except for $\int_\Omega G(x, u) \ dx$, are linear or quadratic forms in the variables $u$ and $\varphi$, and thus are of class $C^1$ (see \cite{BadialeSerra2011}). Moreover, also $\int_\Omega G(x, u) \ dx$ is of class $C^1$, due to the subcritical growth assumption \ref{g2} (see, e.g, \cite[Appendix B]{Rabinowitz1986}). Therefore $F_\omega \in C^1(H_0^1(\Omega) \times \H)$. Straightforward computations yield the following expressions for the partial derivatives:
\begin{align}
    & \frac{\partial F_\omega}{\partial u}(u, \varphi)[v] = \frac{1}{2} \int_\Omega \nabla u \nabla v \ dx + \int_\Omega (\varphi + \chi - \omega) u v  \ dx - \int_\Omega g(x, u) v \ dx, \quad v \in H_0^1(\Omega), \label{eq:F_u} \\
    & \frac{\partial F_\omega}{\partial \varphi}(u, \varphi)[\eta] = \frac{1}{2} \int_\Omega \eta u^2 \ dx - \frac{1}{8\pi} \int_\Omega \nabla \varphi \nabla \eta - \frac{1}{8\pi} \int_\Omega \Delta \varphi \Delta \eta \ dx, \quad \eta \in \H. \label{eq:F_varphi}
\end{align}

From the expressions of the partial derivatives, we readily obtain that
\begin{prop}
    \label{prop:variational_formulation_Dirichlet}
    The pair $(u, \varphi) \in H_0^1(\Omega) \times \H$ is a weak solution to \eqref{eq:modified_system_Dirichlet} if and only if $(u, \varphi)$ is a critical point of $F_\omega$ in $H_0^1(\Omega) \times \H$.
\end{prop}

Observe that the functional is strongly indefinite both from above and from below. More precisely, for every fixed pair $(u, \varphi) \in H_0^1(\Omega) \times \H$, we have $F_\omega(tu, \varphi) \to + \infty$ as $t \to + \infty$, because the gradient term ``wins" against the $L^p$-subcritical terms (due to the Sobolev embeddings). Similarly, $F_\omega(u, t \varphi) \to - \infty$ as $t \to + \infty$, because of the quadratic terms with negative sign.

Due to this fact, standard variational methods do not apply directly, as was already noticed in \cite{BenciFortunato1998}. Indeed, a key idea of the argument presented in \cite{BenciFortunato1998} (and successfully employed by many other authors) is to perform a suitable ``substitution" in the system, thereby reducing the study to an analysis of a functional of a single variable. 

The procedure goes as follows. For each $u \in H_0^1(\Omega)$, we consider the unique weak solution $\Phi_u \in \H$ of the problem
\begin{equation}
    \label{eq:equation_for_Phi_Dirichlet}
    \left\{
    \begin{array}{rcll}
        - \Delta \varphi + \Delta^2 \varphi & = & 4 \pi u^2 & \quad \text{ in } \Omega \\
        \varphi & = & 0 & \quad \text{ on } \po \\
        \Delta \varphi & = & 0 & \quad \text{ on } \po,
    \end{array}
    \right.
    ,
\end{equation}
which can be shown to exist by a slight modification of the argument presented in the proof of Lemma \ref{lem:exitence_auxiliary_problem}. In this way, we can define a map
\begin{equation*}
    \Phi: u \in H_0^1(\Omega) \mapsto \Phi(u) = \Phi_u \in \H.
\end{equation*}
Observe that the map $\Phi$ is implicitly defined by the equation
\begin{equation}
    \label{eq:implicit_def_Phi_Dirichlet}
    \frac{\partial F_\omega}{\partial \varphi} (u, \varphi) = 0, \quad \varphi \in \H,
\end{equation}
which is nothing more than the weak formulation of \eqref{eq:equation_for_Phi_Dirichlet}.

Observe that $\Phi$ is even and $\Phi(0) = 0$. Moreover, we have the following:

\begin{lemma}
    \label{lem:map_Phi}
    The map $\Phi$ is of class $C^1$ and bounded. Moreover, 
    \begin{equation}
        \label{eq:bound_Phi}
        \|\Phi(u)\| \leq C\|u\|^2
    \end{equation}
    for some $C > 0$.
\end{lemma}
\begin{proof}
    To show that $\Phi$ is of class $C^1$, we use the implicit formulation \eqref{eq:implicit_def_Phi_Dirichlet}. Note that the derivatives of $\frac{\partial F_\omega}{\partial \varphi}$ are given by
    \begin{align}
        & \frac{\partial^2 F_\omega}{\partial \varphi \partial u} (u, \varphi)[\eta, v] = \int_\Omega u \eta v \ dx, \nonumber \\
        & \frac{\partial^2 F_\omega}{\partial \varphi^2} (u, \varphi) [\eta, \zeta] = - \frac{1}{8\pi} \left(\int_\Omega \nabla \eta \nabla \zeta \ dx + \int_\Omega \Delta \eta \Delta \zeta \ dx \right),
    \end{align}
    which are readily seen to be continuous. Therefore $\Phi$ is of class $C^1$ (see, e.g., \cite{AmbrosettiProdi1993}).

    Next, we make use of \eqref{eq:equation_for_Phi_Dirichlet} and the Sobolev embeddings to show that $\Phi$ is bounded in $H_0^1(\Omega)$. Indeed, multiplying \eqref{eq:equation_for_Phi_Dirichlet} by $\Phi(u)$ and integrating by parts twice we obtain
    \begin{equation*}
        \int_\Omega |\nabla \Phi(u)|^2 \ dx + \int_\Omega |\Delta \Phi(u)|^2 \ dx = 4 \pi \int_\Omega \Phi(u) u^2 \ dx
    \end{equation*}
    Therefore, by invoking well-known Sobolev embeddings, we have
    \begin{align}
        \|\Phi(u)\|_\H^2
        & \leq 4 \pi \|\Phi(u)\|_2 \|u^2\|_2 \nonumber \\
        & \leq c_1 \|\Phi(u)\|_\H \|u\|_4^2 \nonumber \\
        & \leq c_2 \|\Phi(u)\|_\H \|u\|^2, \nonumber
    \end{align}
    which completes the proof.
\end{proof}

With the map $\Phi$ at hand, we can define the following reduced functional:
\begin{equation*}
    J_\omega : u \in H_0^1(\Omega) \mapsto F_\omega(u, \Phi(u)) \in \mathbb R,
\end{equation*}
which can be written as
\begin{align}
    J_\omega(u)
    & = \frac{1}{4} \int_\Omega |\nabla u|^2 \ dx + \frac{1}{2} \int_\Omega (\chi - \omega) u^2 \ dx - \int_\Omega G(x, u) \ dx \nonumber \\
    & \quad + \frac{1}{16 \pi} \int_\Omega |\nabla \Phi(u)|^2 \ dx + \frac{1}{16 \pi} \int_\Omega |\Delta \Phi(u)|^2 \ dx. \label{eq:def_J}
\end{align}

By making use of the chain rule together with \eqref{eq:implicit_def_Phi_Dirichlet}, we obtain
\begin{align}
    J_\omega'(u)[v]
    & = \frac{\partial F_\omega}{\partial u}(u, \Phi(u))[v] + \left(\frac{\partial F_\omega}{\partial \varphi
    }(u, \Phi(u)) \circ \Phi'(u)\right) [v] \nonumber \\
    & = \frac{1}{2} \int_\Omega \nabla u \nabla v \ dx + \int_\Omega (\Phi(u) + \chi - \omega) u v \ dx - \int_\Omega g(x, u) v \ dx. \label{eq:J_prime}  
\end{align}

The next result tells us that the problem \eqref{eq:modified_system_Dirichlet} can indeed be studied through the functional $J_\omega$.

\begin{prop}
    \label{prop:equivalence_F_J_Dirichlet}
    The pair $(u, \varphi) \in H_0^1(\Omega) \times \H$ is a critical point of $F_\omega$ if and only if $u$ is a critical point of $J_\omega$ and $\varphi = \Phi(u)$.
\end{prop}

\begin{proof}
    Suppose $(u, \varphi)$ is a critical point for $F_\omega$. From \eqref{eq:implicit_def_Phi_Dirichlet}, it follows that $\varphi = \Phi(u)$, and then it follows from \eqref{eq:F_u} that $J_\omega'(u) = 0$.

    Conversely, if $\varphi = \Phi(u)$ then $\frac{\partial F_\omega}{\partial \varphi}(u, \Phi(u)) = 0$ by \eqref{eq:implicit_def_Phi_Dirichlet}, and $\frac{\partial F_\omega}{\partial u}(u, \Phi(u)) = 0$ since $J_\omega'(u) = 0$ (taking into account \eqref{eq:F_u} and \eqref{eq:J_prime}).
\end{proof}

\subsection{Analysis of the reduced functional}

We begin by showing that the reduced functional $J_\omega$ defined in \eqref{eq:def_J} satisfies the Palais-Smale condition.

\begin{lemma}
    \label{lem:PS_condition}
    The functional $J_\omega$ defined in \eqref{eq:def_J} satisfies the Palais-Smale condition.
\end{lemma}

\begin{proof}
    Let $(u_n)_{n \in \mathbb N}$ be a Palais-Smale sequence, that is, 
    \begin{align}
        & |J(u_n)| \leq M \quad \forall n \in \mathbb N, \text{ for some } M > 0, \ \label{eq:PS_a} \\
        & J'(u_n) \to 0 \quad \text{ in } H^{-1} \text{ as } n \to \infty. \label{eq:PS_b} 
    \end{align}

    As usual in this type of argument, the first step is to show that the sequence $(u_n)_{n \in \mathbb N}$ is bounded. To this aim, we take $r$ as in \ref{g4} and make use of \eqref{eq:PS_a} to obtain
    \begin{align}
        \frac{1}{4} \int_\Omega |\nabla u|^2 \ dx 
        & + \frac{1}{2} \int_\Omega (\chi - \omega) u^2 \ dx + \frac{1}{16 \pi} \int_\Omega |\nabla \Phi(u)|^2 \ dx + \frac{1}{16 \pi} \int_\Omega |\Delta \Phi(u)|^2 \ dx \nonumber \\
        & \leq M + \int_{\{x \in \Omega \ : \ |u_n(x)| < r\}} |G(x, u_n)| \ dx + \int_{\{x \in \Omega \ : \ |u_n(x)| \geq r\}} |G(x, u_n)| \ dx \nonumber \\
        & \leq M_1 + \frac{1}{\mu} \int_{\{x \in \Omega \ : \ |u_n(x)| \geq r\}} g(x, u_n) u_n \ dx \nonumber \\
        & \leq M_2 + \frac{1}{\mu} \int_\Omega g(x, u_n) u_n \ dx \quad \forall n \in \mathbb N, \label{eq:estimate_PS_a}
    \end{align}
    where $M_1$ and $M_2$ are suitable positive constants. 

    On the other hand, since
    \begin{align}
        |J_\omega'(u_n)[u_n]|
        & = \left|\frac{1}{2} \int_\Omega |\nabla u_n|^2 \ dx + \int_\Omega (\Phi(u) + \chi - \omega) u_n^2 \ dx - \int_\Omega g(x, u_n) u_n \ dx\right| \nonumber \\
        & \geq \left| \left|\int_\Omega g(x, u_n) u_n \ dx\right| - \left|\frac{1}{2} \int_\Omega |\nabla u_n|^2 \ dx + \int_\Omega (\Phi(u_n) + \chi - \omega) u_n^2 \ dx\right| \right| \nonumber \\
        & \geq \int_\Omega g(x, u_n) u_n \ dx - \left|\frac{1}{2} \int_\Omega |\nabla u_n|^2 \ dx + \int_\Omega (\Phi(u_n) + \chi - \omega) u_n^2 \ dx\right|, \nonumber
    \end{align}
    from \eqref{eq:PS_b} we obtain that there exists some positive constant $M_3$ such that $|J_\omega'(u_n)[u_n]| \leq M_3 \|u_n\|$, and therefore
    \begin{align}
        \int_\Omega g(x, u_n) u_n \ dx
        & \leq M_3 \|u_n\| + \frac{1}{2} \int_\Omega |\nabla u_n|^2 \ dx + \int_\Omega (\Phi(u_n) + \chi - \omega) u_n^2 \ dx \nonumber \\
        & = M_3 \|u_n\| + \frac{1}{2} \|u_n\|^2 + \int_\Omega \chi u_n^2 \ dx - \omega \|u_n\|_2^2 + \frac{1}{4\pi} \|\Phi(u_n)\|^2. \nonumber 
    \end{align}
    Substituting into \eqref{eq:estimate_PS_a}, we obtain
    \begin{align}
        \frac{1}{4} \int_\Omega |\nabla u|^2 \ dx 
        & + \frac{1}{2} \int_\Omega (\chi - \omega) u^2 \ dx + \frac{1}{16 \pi} \int_\Omega |\nabla \Phi(u)|^2 \ dx + \frac{1}{16 \pi} \int_\Omega |\Delta \Phi(u)|^2 \ dx \nonumber \\
        & \leq M_2 + \frac{1}{\mu}\left(M_3 \|u_n\| + \frac{1}{2} \|u_n\|^2 + \int_\Omega \chi u_n^2 \ dx - \omega \|u_n\|_2^2 + \frac{1}{4\pi} \|\Phi(u_n)\|^2 \right) \nonumber \\
        & \leq M_2 + \frac{1}{\mu}\left(M_3 \|u_n\| + \frac{1}{2} \|u_n\|^2 + (\|\chi\|_\infty - \omega) \|u_n\|_2^2 + \frac{1}{4\pi} \|\Phi(u_n)\|^2 \right) \quad \forall n \in \mathbb N. \nonumber
    \end{align}
    Since $\mu > 4$, we have
    \begin{equation*}
        \frac{1}{4\pi \mu} \|\Phi(u_n)\|^2 - \frac{1}{16\pi} \left(\int_\Omega |\nabla \Phi(u_n)|^2 \ dx + \int_\Omega |\Delta \Phi(u_n)|^2 \ dx \right) < 0,
    \end{equation*}
    and therefore we have
    \begin{equation}
        \label{eq:estimate_PS_b}
        \frac{\mu - 2}{4\mu} (\|u_n\|^2 - (\|\chi\|_\infty - \omega)\|u_n\|_2^2) \leq M_2 + \frac{M_3}{\mu} \|u_n\| \quad \forall n \in \mathbb N.
    \end{equation}

    If $\|\chi\|_\infty - \omega \leq 0$, we readily obtain that $(u_n)_{n \in \mathbb N}$ is bounded. If, instead, it holds $\|\chi\|_\infty - \omega > 0$, then from \eqref{eq:estimate_PS_b} we infer that
    \begin{align}
        \|u_n\|_2^2
        & \geq \frac{c_3}{\|\chi\|_\infty - \omega} \left( \frac{\mu - 2}{4\mu}\|u_n\|^2 - \frac{M_3}{\mu}\|u_n\| - M_2 \right) \nonumber \\
        & \geq c_4 \|u_n\|^2 - c_5\|u_n\| - c_6 \quad \forall n \in \mathbb N. \nonumber
    \end{align}
    Now, should the sequence $(u_n)_{n \in \mathbb N}$ be unbounded in $H_0^1(\Omega)$, we would obtain $\|u_n\|_2^2 \to + \infty$ as $n \to \infty$. However, from assumption \ref{g4} we deduce that there exist constants $b_1, b_2 > 0$ such that for any $x \in \overline \Omega$ and $\xi \in \mathbb R$ it holds
    \begin{equation*}
        G(x, \xi) \geq b_1 |\xi|^\mu - b_2.
    \end{equation*}
    But then, since $\Phi$ is bounded, we have
    \begin{align}
        J_\omega(u_n)
        & \leq \frac{1}{4}\|u_n\|^2 + \frac{1}{2}(\|\chi\|_\infty - \omega) \|u_n\|_2^2 - b_1 \int_\Omega \|u_n\|^\mu \ dx + b_2 |\Omega| + \frac{1}{16\pi}\|\Phi(u_n)\|^2 \nonumber \\
        & \leq c_7\|u_n\|^2 + c_8 \|u_n\|_2^2 - b_1 \int_\Omega |u_n|^\mu \ dx + b_2|\Omega| + c_9\|u_n\|_2^2 \nonumber \\
        & \to - \infty \quad \text{ as } n \to \infty, \nonumber 
    \end{align}
    since $\mu > 4$. This contradicts \eqref{eq:PS_a}, and therefore we conclude that $(u_n)_{n \in \mathbb N}$ is bounded also in case $\|\chi\|_\infty - \omega > 0$.

    Since $(u_n)_{n \in \mathbb N}$ is bounded, then there exists $u \in H_0^1(\Omega)$ such that
    \begin{equation*}
        u_n \rightharpoonup u \quad \text{ weakly in } H_0^1(\Omega) \text{ as } n \to \infty.
    \end{equation*}
    We now proceed to show that the convergence is, in fact, strong in $H_0^1(\Omega)$. 

    To this aim, we apply \eqref{eq:J_prime} to obtain
    \begin{equation}
        \label{eq:PS_J_prime}
        - \frac{1}{2} \Delta u_n = J_\omega'(u_n) - \Phi(u_n)u_n + (\chi - \omega) u_n + g(x, u_n),
    \end{equation}
    where this equality is understood in $H^{-1}(\Omega)$. Now, since the resolvent operator $(- \Delta)^{-1}: H^{-1}(\Omega) \to H_0^1(\Omega)$ is compact, to conclude our proof it suffices to show that the right-hand side in \eqref{eq:PS_J_prime} is bounded.

    By assumption, $J_\omega'(u_n) \to 0$ as $n \to \infty$, so $\{J_\omega'(u_n)\}_{n \in \mathbb N}$ is bounded. Since $g$ is continuous and has subcritical growth, the Nemitski operator
    \begin{equation*}
        u \in H_0^1(\Omega) \mapsto g(x, u) \in H^{-1}(\Omega)
    \end{equation*}
    is compact (see, for example, \cite{AmbrosettiMalchiodi2007, AmbrosettiProdi1993}) and therefore, since $(u_n)_{n \in \mathbb N}$ is bounded, also $\{g(x, u_n)\}_{n \in \mathbb N}$ converges in $H^{-1}(\Omega)$. Furthermore, since $H_0^1(\Omega)$ is reflexive, then $H^{-1}(\Omega)$ is reflexive. Since $(\chi - \omega) u_n \to (\chi - \omega) u$ in the weak-$^*$ topology of $H^{-1}(\Omega)$, then $\{(\chi - \omega)u_n\}_{n \in \mathbb N}$ is bounded in $H^{-1}(\Omega)$. Finally, Hölder's inequality together with \eqref{eq:bound_Phi} yields
    \begin{align}
        \|\Phi(u_n) u_n\|_{3/2}^{3/2}
        & \leq \|\Phi(u_n)\|_3^{3/2} \|u_n\|_3^{3/2} \nonumber \\
        & \leq c_{10} \|\Phi(u_n)\|^{3/2} \|u_n\|_3^{3/2} \nonumber \\
        & \leq c_{11} \|u_n\|^3 \|u_n\|_3^{3/2}.
    \end{align}
    Hence $\{\Phi(u_n)u_n\}_{n \in \mathbb N}$ is bounded in $L^{3/2}(\Omega)$, which is continuously embedded into $H^{-1}(\Omega)$, because $H_0^1(\Omega)$ is continuously embedded into $L^3(\Omega)$.

    Hence all terms in the right-hand side of \eqref{eq:PS_J_prime} are bounded, and since $(- \Delta)^{-1}$ is compact, it follows that the sequence $(u_n)_{n \in \mathbb N}$ converges strongly in $H_0^1(\Omega)$.
\end{proof}

\begin{prop}
    \label{prop:Critical_points_J_Dirichlet}
    For any $\omega \in \mathbb R$, the functional $J_\omega$ defined in \eqref{eq:def_J} has infinitely many critical points in $H_0^1(\Omega)$.
\end{prop}

\begin{proof}
    Our aim is to apply Theorem \ref{thm:Mountain_Pass} to the functional $J_\omega$. Since $J_\omega$ is even (because so is $\Phi$), $J_\omega(0) = 0$ and $J_\omega$ satisfies the Palais-Smale condition (Lemma \ref{lem:PS_condition}), it remains only to prove that the geometrical conditions of Theorem \ref{thm:Mountain_Pass} hold.

    We begin by recalling from the proof of Lemma \ref{lem:PS_condition} that
    \begin{align}
        J_\omega(u) 
        & \leq \frac{1}{4}\|u\|^2 + \frac{1}{2}(\|\chi\|_\infty - \omega) \|u\|_2^2 - b_1 \int_\Omega \|u\|^\mu \ dx + b_2 |\Omega| + \frac{1}{16\pi}\|\Phi(u)\|^2 \nonumber \\
        & \to - \infty \nonumber
    \end{align}
    as $\|u\| \to + \infty$, since $\mu > 4$. Therefore, condition \textit{(ii)} in Theorem \ref{thm:Mountain_Pass} is satisfied.

    Now, suppose that
    \begin{equation}
        \label{eq:case_1_Dirichlet}
        \|\chi\|_\infty + \omega < \frac{1}{2} \lambda_1,
    \end{equation}
    where $\lambda_1$ is the first eigenvalue of $- \Delta$ in $H_0^1(\Omega)$ (see Section \ref{sec:prelim}). In this case, we can take $V = \{0\}$ and $X = H_0^1(\Omega)$. Indeed, $J_\omega$ has a strict local minimum at $u = 0$. This can be proven as follows. From \ref{g2} and \ref{g3} we deduce that for every $\varepsilon > 0$ there exists $C_\varepsilon > 0$ such that for any $x \in \overline \Omega$ and $\xi \in \mathbb R$ it holds
    \begin{equation*}
        |G(x, \xi)| \leq \frac{\varepsilon}{2} \xi^2 + C_\varepsilon |\xi|^p,
    \end{equation*}
    for $p \in (4, 6)$.

    Let us fix a number $c$ such that $\|\chi\|_\infty + \omega < c < \frac{1}{2} \lambda_1$. Since $H_0^1(\Omega) \hookrightarrow L^p(\Omega)$, by Poincaré inequality we obtain
    \begin{align}
        J_\omega(u)
        & \geq \frac{1}{4} \int_\Omega |\nabla u|^2 \ dx - \frac{1}{2} \|\chi\|_\infty \int_\Omega u^2 \ dx - \frac{1}{2} \omega \int_\Omega u^2 \ dx - \int_\Omega G(x, u) \ dx \nonumber \\
        & \geq \frac{1}{4}(\|u\|^2 - 2(\|\chi\|_\infty + \omega) \|u\|_2^2) - \frac{\varepsilon}{2} \|u\|_2^2 - C_\varepsilon \|u\|_p^p \nonumber \\
        & > \frac{1}{4} \frac{\lambda_1 - c}{\lambda_1}\|u\|^2 - \frac{\varepsilon}{2 \lambda_1} \|u\|^2 - C_\varepsilon' \|u\|^p \quad \forall u \in H_0^1(\Omega), \nonumber
    \end{align}
    which is positive for $\|u\| = \rho$ if $\rho$ is sufficiently small. So condition \textit{(i)} of Theorem \ref{thm:Mountain_Pass} is satisfied in the case when \eqref{eq:case_1_Dirichlet} holds.

    Next, we assume that 
    \begin{equation}
        \label{eq:case_2_Dirichlet}
        \frac{1}{2} \lambda_1 \leq \|\chi\|_\infty + \omega
    \end{equation}
    and let $c \in \left(\frac{1}{2} \lambda_1, \|\chi\|_\infty + \omega \right)$. Set 
    \begin{equation*}
        k_\omega \coloneqq \min\{k \in \mathbb N \ : \ \|\chi\|_\infty + \omega < \lambda_k\}, 
    \end{equation*}
    and consider the following splitting of $H_0^1(\Omega)$:
    \begin{equation*}
        H_0^1(\Omega) = V_k \oplus X,
    \end{equation*}
    where
    \begin{equation*}
        V_k = \bigoplus_{k = 1}^{k_\omega - 1} H_k, \quad X = V_k^\perp = \overline{\bigoplus_{k = k_\omega}^\infty H_k}.
    \end{equation*}
    Then, since
    \begin{equation*}
        \|v\|^2 \geq \lambda_{k_\omega} \|v\|^2 \quad \forall v \in X,
    \end{equation*}
    (by the variational characterization of the eigenvalues of $- \Delta$ in $H_0^1(\Omega)$), we have
    \begin{align}
        J_\omega(u)
        & \geq \frac{1}{4} \int_\Omega |\nabla u|^2 \ dx - \frac{1}{2} \|\chi\|_\infty \int_\Omega u^2 \ dx - \frac{1}{2} \omega \int_\Omega u^2 \ dx - \int_\Omega G(x, u) \ dx \nonumber \\
        & \geq \frac{1}{4}(\|u\|^2 - 2(\|\chi\|_\infty + \omega) \|u\|_2^2) - \frac{\varepsilon}{2} \|u\|_2^2 - C_\varepsilon \|u\|_p^p \nonumber \\
        & > \frac{1}{4} \frac{\lambda_{k_\omega} - c}{\lambda_{k_\omega}}\|u\|^2 - \frac{\varepsilon}{2 \lambda_1} \|u\|^2 - C_\varepsilon' \|u\|^p \quad \forall u \in X, \nonumber
    \end{align}
    so that $J_\omega$ is positive in small spheres of $X$.

    We can therefore apply Theorem \ref{thm:Mountain_Pass} to conclude the proof.
\end{proof}

\begin{theorem}
    \label{thm:main_Dirichlet}
    Let $\Omega \subset \mathbb R^3$ be a smooth, bounded domain. If $g \in C(\overline \Omega \times \mathbb R)$ satisfies \ref{g1}-\ref{g4}, then for every triple $(h_1, h_2, \omega) \in C(\po) \times C(\po) \times \mathbb R$ there exist infinitely many weak solutions $(u, \phi) \in H_0^1(\Omega) \times H^2(\Omega)$ to the problem
    \begin{equation*}
    \left\{
    \begin{array}{rcll}
        - \frac{1}{2} \Delta u + \phi u - g(x, u) & = & \omega u & \quad \text{ in } \Omega \\
        - \Delta \phi + \Delta^2 \phi & = & 4 \pi u^2 & \quad \text{ in } \Omega \\
        u & = & 0 & \quad \text{ on } \po \\
        \phi & = & h_1 & \quad \text{ on } \po \\
        \Delta \phi & = & h_2 & \quad \text{ on } \po \\
    \end{array}
    \right.
    .
\end{equation*}
\end{theorem}
\begin{proof}
    Follows by combining Propositions \ref{prop:variational_formulation_Dirichlet}, \ref{prop:Critical_points_J_Dirichlet}, and \ref{prop:Critical_points_J_Dirichlet}.
\end{proof}

\begin{remark}
    The physical meaning of Theorem \ref{thm:main_Dirichlet} is that, as happens in the case of Maxwell electrodynamics (\cite{PisaniSiciliano2007b}), if the Schrödinger equation is perturbed with a suitable nonlinearity, then for any prescribed frequency $\omega$ there exist (infinitely many) stationary solutions with that frequency.
\end{remark}

\section{Non-existence of solutions for the unperturbed problem}
\label{sec:nonexistence}
\begin{lemma}
    \label{lem:max_principle_Dirichlet}
    Let $h_1, h_2 \in C(\po)$ be such that $h_1 - h_2 \geq 0$ on $\po$. If $\phi \in H_2(\Omega)$ satisfies
    \begin{equation*}
        \left\{
        \begin{array}{rcll}
            - \Delta \phi + \Delta^2 \phi & \geq & 0 & \quad \text{ in } \Omega \\
            \phi & = & h_1 & \quad \text{ on } \po \\
            \Delta \phi & = & h_2 & \quad \text{ on } \po,
        \end{array}
        \right.
    \end{equation*}
    in the weak sense, then $\phi > 0$ in $\Omega$.
\end{lemma}

\begin{proof}
    Notice that we have
    \begin{equation*}
        - \Delta (\phi - \Delta \phi) \geq 0 \quad \text{ in } \Omega,
    \end{equation*}
    so that the maximum principle yields
    \begin{equation*}
        \phi - \Delta \phi \geq 0.
    \end{equation*}
    Another application of the maximum principle yields the claim.
\end{proof}

\begin{prop}
    If $h_1 - h_2 \geq 0$ and $\omega < \frac{\lambda_1}{2}$, then there are no solutions to the problem
    \begin{equation*}
    \left\{
    \begin{array}{rcll}
        - \frac{1}{2} \Delta u + \phi u & = & \omega u & \quad \text{ in } \Omega \\
        - \Delta \phi + \Delta^2 \phi & = & 4 \pi u^2 & \quad \text{ in } \Omega \\
        u & = & 0 & \quad \text{ on } \po \\
        \phi & = & h_1 & \quad \text{ on } \po \\
        \Delta \phi & = & h_2 & \quad \text{ on } \po \\
    \end{array}
    \right.
    .
\end{equation*}
\end{prop}
\begin{proof}
    By Lemma \ref{lem:max_principle_Dirichlet}, $\varphi > 0$ in $\Omega$. Multiplying the first equation by $u$ and integrating by parts, we obtain
    \begin{equation*}
        \omega \int_\Omega u^2 \ dx > \frac{1}{2} \int_\Omega |\nabla u|^2 \ dx,
    \end{equation*}
    which, by the variational characterization of the eigenvalues of $- \Delta$ in $H_0^1(\Omega)$, can only hold if $\omega > \frac{\lambda_1}{2}$.
\end{proof}

\section*{Acknowledgements}
This work has been partially funded by the European Union - NextGenerationEU within the framework of PNRR  Mission 4 - Component 2 - Investment 1.1 under the Italian Ministry of University and Research (MUR) program PRIN 2022 - Grant number 2022BCFHN2 - Advanced theoretical aspects in PDEs and their applications - CUP: H53D23001960006.

This research was also partially supported by Gruppo Nazionale per l'Analisi Matematica, la Probabilità e le loro Applicazioni (GNAMPA) of the Istituto Nazionale di Alta Matematica (INdAM).

\bibliographystyle{acm}
\bibliography{ref_math}

\end{document}